\documentclass[11pt,a4paper,amsfonts]{amsart}

\usepackage{amsmath}
\usepackage{amssymb,amscd,amsxtra,calc}
\usepackage{mathrsfs}
\usepackage{mathtools}
\usepackage{enumitem}
\usepackage{tikz-cd}
\usepackage{array}
\usepackage{booktabs}
\usepackage[colorlinks,linkcolor=blue,anchorcolor=blue,citecolor=green,backref=none]{hyperref}

\theoremstyle{plain}
\newtheorem{thm}{Theorem}[section]

\newtheorem{theorem}[thm]{Theorem}
\newtheorem{proposition}[thm]{Proposition}
\newtheorem{lemma}[thm]{Lemma}
\newtheorem{corollary}[thm]{Corollary}

\theoremstyle{definition}
\newtheorem{definition}[thm]{Definition}
\newtheorem{example}[thm]{Example}
\newtheorem{question}[thm]{Question}
\newtheorem{remark}[thm]{Remark}

\newcommand{\A}{\mathbb{A}}

\newcommand{\Gm}{\mathbb{G}_m}

\newcommand{\Z}{\mathbb{Z}}
\newcommand{\SO}{\mathcal{O}}

\newcommand{\Spec}{\operatorname{Spec}}
\newcommand{\Frac}{\operatorname{Frac}}

\newcommand{\Hom}{\operatorname{Hom}}

\newcommand{\rank}{\operatorname{rank}}

\newcommand{\ed}{\operatorname{ed}}
\newcommand{\trdeg}{\operatorname{trdeg}}

\newcommand{\id}{\operatorname{id}}

\newcommand{\ST}{\operatorname{ST}}

\title[Galois self-covers of projective spaces]{Galois self-covers of projective spaces and essential dimensions}
\author{Yujie Luo and De-Qi Zhang}
\date{}

\address{Department of Mathematics, National University of Singapore, Singapore 119076, 
Singapore}

\email{yujieluo96@gmail.com~and~lyj96@nus.edu.sg}

\address{Department of Mathematics, National University of Singapore, Singapore 119076, 
Singapore}

\email{matzdq@nus.edu.sg}


\begin{document}

\begin{abstract}
We give the structure theorem of Galois self-covers $f: \mathbf{P}^n \to \mathbf{P}^n$. As an application, we show that the essential dimension of every such nontrivial cover attains its maximum possible value $n$. As another application, we prove that the pair $(\mathbf{P}^n, R_f/(q-1))$ is log Calabi-Yau as conjectured by Gongyo, where $R_f$ is the ramification divisor and we write $f^*\mathcal{O}(1) = \mathcal{O}(q)$.
\end{abstract}

\maketitle
\tableofcontents

\section{Introduction}

Throughout this article, we work over an algebraically closed field $k$ of characteristic zero.

\medskip

A surjective morphism $f:\mathbf{P}^n\to \mathbf{P}^n$ is called a {\it Galois cover} if the corresponding extension of function fields is Galois with Galois group $G$ (then $G$ acts biregularly on $\mathbf{P}^n$ with quotient being isomorphic to $\mathbf{P}^n$). For such a morphism, there is a unique integer $q\geq 1$ such that $f^*\mathcal{O}_{\mathbf{P}^n}(1)\cong \mathcal{O}_{\mathbf{P}^n}(q)$. In particular, $\deg(f)=q^n$. In this paper, we consider the nontrivial case $q>1$ and give a complete description of Galois self-covers of projective spaces. 

Recall that a linear transformation in $\mathrm{GL}(n,k)$ is called a pseudoreflection if it fixes a hyperplane pointwise. Furthermore, a finite subgroup of $\mathrm{GL}(n,k)$ is called a pseudoreflection group if it is generated by pseudoreflections. Shephard and Todd \cite{ST54} provided a complete classification of pseudoreflection groups over the complex numbers. By the Lefschetz principle, the same classification works over algebraically closed field $k$ of characteristic zero. Our structure theorem below asserts that Galois groups of Galois self-covers of projective spaces are constructed from pseudoreflection groups of rank at most two.

\begin{theorem}\label{thm:structure-intro}
Let $f:\mathbf{P}^n\to \mathbf{P}^n$ be a surjective morphism with $f^*\SO_{\mathbf{P}^n}(1)\cong \SO_{\mathbf{P}^n}(q)$ and $q>1$. Suppose that $f$ is Galois with Galois group $G$. Then after linear changes of coordinates on the source and the target, $f$ satisfies the following.

\begin{enumerate}
\item[(1)]
There is a direct sum decomposition $k^{n+1}=V_1\oplus \cdots \oplus V_s$ and finite irreducible pseudoreflection groups $\Gamma_i\subset \mathrm{GL}(V_i)$ such that every basic invariant degree of every $\Gamma_i$ is equal to $q$. 
\item[(2)]
Set $k[V_i]^{\Gamma_i}=k[\Phi_{i,1},\ldots,\Phi_{i,r_i}]$, where $r_i=\dim V_i$ and $\deg \Phi_{i,j}=q$, and $v=v_1+\cdots+v_s$ with $v_i\in V_i$. Then $$f([v])=[\Phi_{1,1}(v_1):\cdots:\Phi_{1,r_1}(v_1):\cdots: \Phi_{s,1}(v_s):\cdots:\Phi_{s,r_s}(v_s)].$$
\item[(3)]
The Galois group $G$ of $f$ is $(\Gamma_1\times\cdots\times \Gamma_s)/\Delta\mu_q,$ where $\Delta\mu_q= \{(\lambda I_{V_1},\ldots,\lambda I_{V_s})\mid \lambda^q=1\}.$ 
\end{enumerate}

Conversely, every map $f$ constructed in blocks as in (2) above and satisfying (1) is a finite Galois cover of $\mathbf{P}^n$.
\end{theorem}

\begin{remark}\label{rem: irred classif}
$ $
\begin{enumerate}
\item[(1)]
In Theorem \ref{thm:structure-intro}, the possible irreducible blocks and the corresponding pseudoreflection groups are precisely those listed in the following table,
$$
\begin{array}{c|c|c}
\toprule
\dim V_i & \Gamma_i & q \\
\midrule
1 & \mu_q & q>1 \\
2 & G(q,2,2) & q\geq 4\text{ even} \\
2 & G_7 & q=12 \\
2 & G_{11} & q=24 \\
2 & G_{19} & q=60 \\
\bottomrule
\end{array}
$$
where $\mu_q$ is the cyclic group of order $q$, $G(q,2,2)$ is the imprimitive Shephard--Todd group and $G_7,G_{11},G_{19}$ are the exceptional rank-two Shephard--Todd groups. 

\item[(2)]
When $n=1$, Theorem~\ref{thm:structure-intro} gives the classical classification of finite subgroups of $\mathrm{PGL}_2(k)$, namely
$C_m, D_m, A_4, S_4, A_5.$ Indeed, for every finite subgroup $\Gamma\subset \mathrm{PGL}_2(k)$, the quotient $\mathbf P^1\to \mathbf P^1/\Gamma\cong \mathbf P^1$ is a Galois self-cover.
\end{enumerate}
\end{remark}
\medskip

Recall that $f$ is called a {\it $q$-power map} if, after choosing suitable homogeneous coordinates, $f([x_0:\cdots:x_n])=[x_0^q:\cdots:x_n^q].$ As a consequence of Theorem \ref{thm:structure-intro}, we have the following. 

\begin{corollary}\label{cor:odd-degree-intro}
Let $f:\mathbf{P}^n\to \mathbf{P}^n$ be a finite Galois cover with Galois group $G$. Suppose that one of the following conditions holds.
\begin{enumerate}
\item $\deg f$ is odd.
\item $n\geq 2$ and $G$ is abelian.
\item 
$n \ge 2$, $G$ is nilpotent and it is not a $2$-group. 
\end{enumerate}
Then $f$ is a power map.
\end{corollary}

In view of Example \ref{ex: nilpotent galois cover}, each of the conditions in Corollary \ref{cor:odd-degree-intro} seems optimal. 

In Appendix \ref{Appendix}, we also provide an alternative proof of Corollary~\ref{cor:odd-degree-intro}(2) which is independent of the classification result of Shephard and Todd \cite{ST54}.

\medskip

Let $X$ be a normal variety. A log pair $(X,B)$ is called a \emph{log Calabi--Yau} pair if $(X,B)$ is \emph{log canonical} (cf. \cite{KM98}) and $K_X+B\sim_\mathbb{Q} 0$. For a finite morphism $f:X\to Y$ between smooth varieties, we denote by $R_f$ the ramification divisor defined by $$K_X=f^*K_Y+R_f.$$
As another application of our main theorem, we confirm a conjecture of Gongyo (cf. \cite[Conjecture~1.2]{BG17} and \cite[Conjecture~1.3]{Meng23}) for Galois self-covers on $\mathbf{P}^n$:

\begin{theorem}\label{thm:log-cy}
Let $f:\mathbf{P}^n\to \mathbf{P}^n$ be a finite Galois cover with $f^*\mathcal{O}_{\mathbf{P}^n}(1)\cong \mathcal{O}_{\mathbf{P}^n}(q)$ for some integer $q>1$. Let $R_f$ be the ramification divisor of $f$.  Then $(\mathbf{P}^n,\frac{1}{q-1}R_f)$ is a log Calabi--Yau pair.
\end{theorem}

The following example shows that the above theorem does not hold if we do not assume that $f$ is Galois.

\begin{example}
    We consider the endomorphism $f$ on $\mathbf{P}^2$ such that $$f:[x:y:z] \mapsto [x^3+y^3:x(y^2+z^2):z(3y^2+2z^2)].$$ Then $f^*\mathcal{O}(1)=\mathcal{O}(3)$. A simple calculation shows that the Jacobian $$J_{f}=9y^2(2x^3y-y^4-3y^2z^2-2z^4).$$ Therefore, we have $R_f:=\mathrm{div}(J_f)=2L+C$, where $L:=\{y=0\}$ and $C=\{2x^3y-y^4-3y^2z^2-2z^4=0\}$. Since the ramification index along $C$ is $2$ which does not divide $\mathrm{deg}(f)=9$, $f$ is not Galois. We consider the pair $(\mathbf{P}^2,\frac{1}{2}R_f)$. Since $(\mathbf{P}^2,\frac{1}{2}R_f)|_L=(L,2o)$ for $o=[1:0:0]$, it follows that $(\mathbf{P}^2,\frac{1}{2}R_f)|_L$ is not log canonical. By the inversion of adjunction (cf. \cite[Theorem~5.50]{KM98}), $(\mathbf{P}^2,\frac{1}{2}R_f)$ is not log canonical and consequently not log Calabi-Yau.
\end{example}

\medskip

Now we study essential dimensions of Galois self-covers of $\mathbf{P}^n$. \emph{Essential dimension} was introduced in a modern context by Buhler and Reichstein \cite{BR97, BR99}, and it has been studied for decades (see \cite{Reic10} for more on its history). Recall that the \emph{essential dimension of a generically finite (dominant rational) map} $f: X \dasharrow Y$, denoted by $\mathrm{ed}(f)$, is the smallest integer $d$ such that $f$ is birational to the pullback of a map of varieties of dimension $d$ (see \cite{BR97} and Definition~\ref{defn: ess}). A generically finite map $f: X\dasharrow Y$ is called \emph{incompressible} if its essential dimension attains the maximum possible value $\dim(X)$ (cf. \cite{BR97,FKW24,KZ}).

It is known that the $p$-th power maps on toric varieties are even $p$-incompressible for all $p$, see \cite[Theorem~1.2]{RY00} and \cite[Proposition~10]{BF24}. Hence, a direct consequence of Corollary~\ref{cor:odd-degree-intro} says that non-trivial abelian self-covers of projective spaces are incompressible. In this paper (see Section \ref{sect:ess dim}), we prove the following.

\begin{theorem}\label{thm:ed-intro}
Let $f:\mathbf{P}^n\to \mathbf{P}^n$ be a finite Galois cover with $\deg(f)>1$. Then $\ed(f)=n$.
\end{theorem}

This answers a question of Koll\'ar and Zhuang \cite[Question~19]{KZ} affirmatively for Galois self-covers of projective spaces. We remark that Theorem~\ref{thm:ed-intro} does not hold in positive characteristic; see Example~\ref{ex: positive char counterex}. 

Motivated by Theorem~\ref{thm:ed-intro}, it is interesting to ask the following question:

\begin{question}\label{ques: incompressibility of endomorphisms on Pn}
    Let $f$ be a surjective endomorphism on $\mathbf{P}^n$ such that $\deg(f)>1$. Is $f$ incompressible?
\end{question}

We end with the following remark, which answers Question~\ref{ques: incompressibility of endomorphisms on Pn} in the affirmative for the generic maps.

\begin{remark}
Let $\mathrm{End}(\mathbf{P}^n,d)$ be the parameter space of surjective endomorphisms of $\mathbf{P}^n$ of degree $d$. We remark that, by applying Reichstein and Scavia's valuative semicontinuity theorem for the essential dimension of $G$-torsors \cite[Theorem~1.2]{RS22}, one obtains valuative semicontinuity of the essential dimension of surjective endomorphisms in $\mathrm{End}(\mathbf{P}^n,d)$. In particular, the generic surjective endomorphism of degree $d>1$ is incompressible; here by generic we mean it corresponds to the generic point of $\mathrm{End}(\mathbf{P}^n,d)$. Moreover, together with Fakhruddin's result \cite[Theorem~1.2]{Fak14}, which states that a generic surjective endomorphism of $\mathbf{P}^n$ of degree
$d>1$ satisfies the dynamical Manin--Mumford conjecture, it answers a question posed by the authors of this paper and Oguiso \cite[Question~6.8]{LOZ25} for generic surjective endomorphisms of $\mathbf{P}^n$.
\end{remark}

\vskip 1pc
{\bf Acknowledgement.}
The authors are supported, respectively, by the Peng Tsu Ann Assistant Professorship, and the ARF: A-8002487-00-00, of NUS.
\section{Preliminaries}

We adopt the standard notation as in \cite{Ful}, \cite{Har}, \cite{Laz} and \cite{Mum74}.

\subsection{Reflection groups}

\begin{definition}\label{def:reflection}
    Let $V$ be a finite-dimensional vector space over $k$. A nontrivial element $\sigma\in \mathrm{GL}(V)$ of finite order is called a pseudoreflection if it fixes a hyperplane in $V$ pointwise. A finite subgroup $\Gamma\subset \mathrm{GL}(V)$ is called a pseudoreflection group if it is generated by pseudoreflections.
\end{definition}

We shall use the following two standard theorems.

\begin{theorem}[Chevalley--Shephard--Todd Theorem \cite{ST54, Che55}]\label{thm:CST}
    Let $V$ be a vector space of dimension $n$ and let $\Gamma\subset \mathrm{GL}(V)$ be a finite group. Then $k[V]^\Gamma$ is a polynomial algebra if and only if $\Gamma$ is generated by pseudoreflections. If $k[V]^\Gamma=k[F_1,\ldots,F_n]$ with $F_i$ homogeneous of degrees $d_i$, then $|\Gamma|=d_1\cdots d_n.$ 
\end{theorem}

The $d_i = \deg F_i$ ($1 \le i \le n)$ are called the invariant degrees of $\Gamma$ and are, up to re-ordering, independent of the choice of the generators $F_i$ of the invariant ring (cf. \cite[II.~5.1]{ST54}).

\begin{theorem}[{cf. \cite{ST54}}]\label{thm:equal-degree}
Let $\Gamma\subset \mathrm{GL}(W)$ be an irreducible finite pseudoreflection group whose invariant degrees are all equal to the same integer $q>1$. Then exactly one of the following holds:
\begin{enumerate}[label=\textup{(\arabic*)}]
\item $\dim W=1$ and $\Gamma=\mu_q$;
\item $\dim W=2$, $q\geq 4$ is even, and $\Gamma=G(q,2,2)$;
\item $\dim W=2$ and $(\Gamma,q)$ is one of $(G_7,12)$, $(G_{11},24)$, $(G_{19},60)$.
\end{enumerate}
\end{theorem}

\begin{proof}
The Shephard--Todd classification \cite{ST54} consists of the imprimitive groups $G(m,p,r)$ and the primitive exceptional groups. For $G(m,p,r)$, the invariant degrees are $m,2m,\ldots,(r-1)m,\frac{rm}{p}.$ If all degrees are equal, then $r=2$ and $m=\frac{2m}{p},$ so $p=2$. Thus one gets $G(m,2,2)$, with common degree $q=m$. Since $p\mid m$, the integer $q$ is even. For $m=2$, the group $G(2,2,2)$ is reducible, so the irreducible imprimitive cases are $G(q,2,2)$ with $q\geq 4$ even.

For the primitive rank-two groups, the invariant degrees of the Shephard--Todd groups $G_4,\ldots,G_{22}$ are listed in the following table:
$$
\begin{array}{c|c@{\qquad}c|c@{\qquad}c|c}
\ST & (d_1,d_2) & \ST & (d_1,d_2) & \ST & (d_1,d_2)\\
\hline
4&(4,6) & 5&(6,12) & 6&(4,12)\\
7&(12,12)& 8&(8,12) & 9&(8,24)\\
10&(12,24)&11&(24,24)&12&(6,8)\\
13&(8,12)&14&(6,24)&15&(12,24)\\
16&(20,30)&17&(20,60)&18&(30,60)\\
19&(60,60)&20&(12,30)&21&(12,60)\\
22&(12,20)& & & &
\end{array}
$$
The equal-degree entries are exactly $G_7$, $G_{11}$ and $G_{19}$, with common degrees $12$, $24$ and $60$, respectively.

For primitive exceptional groups of rank at least three, the same table gives unequal degree lists:
$$\begin{array}{c|c|c@{\qquad\qquad}c|c|c}
\mathrm{ST} & \dim W & \text{degrees}
&
\mathrm{ST} & \dim W & \text{degrees}\\
\hline
23 & 3 & 2,6,10
&
31 & 4 & 8,12,20,24\\
24 & 3 & 4,6,14
&
32 & 4 & 12,18,24,30\\
25 & 3 & 6,9,12
&
33 & 5 & 4,6,10,12,18\\
26 & 3 & 6,12,18
&
34 & 6 & 6,12,18,24,30,42\\
27 & 3 & 6,12,30
&
35 & 6 & 2,5,6,8,9,12\\
28 & 4 & 2,6,8,12
&
36 & 7 & 2,6,8,10,12,14,18\\
29 & 4 & 4,8,12,20
&
37 & 8 & 2,8,12,14,18,20,24,30\\
30 & 4 & 2,12,20,30
&
& & 
\end{array}$$
This finishes the proof.
\end{proof}

We shall also use the elementary fact that affine finite quotient maps separate orbits.

\begin{lemma}\label{lem:quotient-fibers}
Let $\Gamma$ be a finite group acting linearly on a vector space $V$, and let $\pi:V\to \Spec k[V]^\Gamma$ be the affine quotient. Then the set of geometric points of every fiber of $\pi$ is a single $\Gamma$-orbit.
\end{lemma}

\begin{proof}
Note that $\pi^{-1}(x)=\{y\in V\mid F(y)=F(x) \text{ for all }F\in k[V]^\Gamma\subset k[V]\}$. Since $\Gamma$-invariant functions are constant on $\Gamma$-orbits, every $\Gamma$-orbit is contained in a fiber. Conversely, let $x,y\in V$ have distinct $\Gamma$-orbits. Then the two finite subsets $\Gamma x$ and $\Gamma y$ are disjoint. Since $V$ is affine, we can find a polynomial $h\in k[V]$ such that $h=0$ on $\Gamma x$ and $h=1$ on $\Gamma y$. We consider $H=\sum_{\gamma\in \Gamma}h\circ \gamma^{-1}$, then $H$ is $\Gamma$-invariant, and $H(x)=0$ while $H(y)=|\Gamma|\neq 0$. Hence $x$ and $y$ do not lie in the same fiber.
\end{proof}

\subsection{Essential dimensions of Kummer extensions}

In this section, we collect some facts about the essential dimension and study the essential dimension of Kummer extensions.

\begin{definition}[cf. {\cite[Definition~2.1]{BR97}}]\label{defn: ess}
    Let $E/F$ be a finite field extension over a base field $k$. We say that $E/F$ is \emph{defined over} a subfield $F_0$ of $F$ if there exists an extension $E_0/F_0$ over $k$ such that $[E_0:F_0]=[E:F]$, $E_0 \subseteq E$ and $E_0F = E$. The \emph{essential dimension} of $E/F$, denoted as $\mathrm{ed}(E/F)$, is the minimal value of $\mathrm{trdeg}_k(F_0)$ as $F_0$ ranges over all fields for which $E/F$ is defined over $F_0$. The \emph{essential dimension of a generically finite dominant rational map} $f: X\dasharrow Y$, denoted by $\mathrm{ed}(f)$, is defined as $\mathrm{ed}_k(K(X)/K(Y))$.
\end{definition}

\begin{lemma}\label{lem: Kummer extension ed}
    Let $d>1$ be an integer, let $a_1,\dots,a_n$ be algebraically independent elements over $k$, and set $$F=k(a_1,\dots,a_n), \qquad E=F(v_1,\dots,v_n), \qquad v_i^d=a_i.$$ 
    Then $\ed(E/F)=n$.
\end{lemma}

\begin{proof}
    Note that $G:=\mathrm{Gal}(E/F)=(\mu_d)^n$. Since $G$ acts linearly and faithfully on $X=\mathrm{Spec}(k[v_1,\dots,v_n])$, we have $\mathrm{ed}(G)=\mathrm{ed}(X)=\mathrm{ed}(E/F)$ by \cite[Lemma~2.7 and Theorem~3.1(b)]{BR97}. By \cite[Theorem~6.1]{BR97}, we have $\mathrm{ed}(G)=n$. This finishes the proof.
\end{proof}

\begin{remark}\label{rem: power maps are incompressible}
    As a direct corollary, all nontrivial power maps are incompressible.
\end{remark}

\section{Proof of {Theorem~\ref{thm:structure-intro} and Remark \ref{rem: irred classif}}}

Let $V=k^{n+1}$, let $S=k[V]$, and let $S_q$ denote the homogeneous elements in $S$ of degree $q$. Let $f:\mathbf{P}(V)\to \mathbf{P}^n$ be a finite Galois cover with $f^*\SO_{\mathbf{P}^n}(1)\cong \SO_{\mathbf{P}(V)}(q)$ for some $q>1.$ Choose $F_0,\ldots,F_n\in S_q$ with no common zero on $\mathbf{P}(V)$ such that $f([v])=[F_0(v):\cdots:F_n(v)].$ Their common affine zero is only the origin. Hence $A:=k[F_0,\ldots,F_{n}]\subset S$ is a polynomial subring over which $S$ is finite. 

Let $G\subset \mathrm{PGL}(V)$ be the Galois group of $f$. Then $|G|=\deg(f)=q^n.$ Let $\mu_q:=\{\lambda I_V\in \mathrm{GL}(V)\mid \lambda^q=1\}$.

\begin{lemma}\label{lem:lifting}
With the above notations, we have the following.
\begin{enumerate}
    \item There is a finite subgroup $\Gamma\subset \mathrm{GL}(V)$ fitting into an exact sequence $$1\to \mu_q\to \Gamma\to G\to 1$$ such that $F_i(\gamma v)=F_i(v)$ for every $\gamma\in \Gamma$ and every $i$. In particular, $|\Gamma|=q^{n+1}.$ 
    \item The affine morphism $$F:V\to \mathbf{A}^{n+1}, \qquad v\longmapsto (F_0(v),\ldots,F_{n}(v))$$ is finite of degree $q^{n+1}$.
    \item $S^\Gamma=A=k[F_0,\ldots,F_{n}].$ In particular, $\Gamma$ is a finite pseudoreflection group and all invariant degrees of $\Gamma$ are equal to $q$.
    \item Conversely, let $\Gamma\subset \mathrm{GL}(V)$ be a finite pseudoreflection group such that $k[V]^\Gamma=k[F_0,\ldots,F_{n}]$ where the $F_i$ are homogeneous of degree $q>1$. Then $\Gamma\cap k^*I_V=\mu_q.$ Moreover, $$f:\mathbf{P}(V)\to \mathbf{P}^{n}, \qquad [v]\longmapsto [F_0(v):\cdots:F_{n}(v)]$$ is a finite Galois morphism of degree $q^n$ whose Galois group is $\Gamma/\mu_q$.
\end{enumerate}

\end{lemma}

\begin{proof}
Take $g\in G$ and choose a lift $\tilde{g}\in \mathrm{GL}(V)$. Since $f\circ g=f$, the two systems $(F_0\circ \tilde{g},\ldots,F_{n}\circ \tilde{g})$ and $(F_0,\ldots,F_{n})$ define the same morphism to projective space. Thus $$F_i(\tilde{g} v)=c(\tilde{g})F_i(v)$$ for all $i$, where $c(\tilde{g})\in k^*$. Choose some $a\in k^*$ such that $a^q c(\tilde{g})=1$. Then $$F_i(a\tilde{g} v)=a^qc(\tilde{g})F_i(v)=F_i(v)$$ for each $i$. It follows that there exists a lifting $a\tilde{g}\in \mathrm{GL}(V)$ of $g$ such that $F_i(a\tilde{g}v)=F_i(v)$ for $0\leq i\leq n$. Define $$\Gamma=\{\gamma\in \mathrm{GL}(V)\mid [\gamma]\in G, \;F_i(\gamma v)=F_i(v)\text{ for all }i\}.$$  Then the natural quotient map $\Gamma\to G$ is surjective by the previous argument. Its kernel consists of scalar matrices $\lambda I_V$ such that $\lambda^qF_i(v)=F_i(v)$ for all $i$, hence $\lambda^q=1$. Thus the kernel is $\mu_q$. In particular, $|\Gamma|=q|G|=q^{n+1}.$ This proves (1).

\medskip 

Now we prove (2). The finiteness of $F$ follows from the fact that $S$ is finite over $A$. Let $y\in \A^{n+1}\setminus\{0\}$ be a general point. The point $[y]\in \mathbf{P}^{n}$ has $q^{n}$ preimages under $f$. If $[v]$ is one such preimage, then a representative $v$ satisfies $(F_0(v),\ldots,F_{n}(v))=c y$ for some $c\in k^*$. The points on the affine line $k v$ mapping to $y$ are the $\lambda v$ with $\lambda^q c=1,$ so there are exactly $q$ such points. Hence a general fiber of $F$ contains $q\cdot q^{n}=q^{n+1}$ points. This proves (2).

\medskip

For (3), since every $F_i$ is $\Gamma$-invariant, one has $A\subset S^\Gamma.$ For a faithful finite linear group in characteristic zero, $[\Frac S:\Frac(S^\Gamma)]=|\Gamma|=q^{n+1}.$ By Lemma~\ref{lem:lifting}(2), $[\Frac S:\Frac A]=q^{n+1}.$ Thus $\Frac A=\Frac(S^\Gamma).$ The ring $S^\Gamma$ is integral over $A$, because $S$ is finite over $A$. Since $A$ is a polynomial ring, it is integrally closed in its fraction field. Hence $S^\Gamma\subset A$. Therefore $S^\Gamma=A$.

Now $S^\Gamma$ is a polynomial algebra. By Theorem~\ref{thm:CST}, $\Gamma$ is generated by pseudoreflections. Since $S^\Gamma$ is generated by $n+1$ homogeneous forms of degree $q$, all invariant degrees of $\Gamma$ are equal to $q$. This proves (3).

\medskip

For (4), note that each $F_i$ is $\mu_q$-invariant. Hence the orbit $\mu_q v$ maps to a single point under the affine quotient morphism $V\to \Spec k[V]^\Gamma.$ By Lemma~\ref{lem:quotient-fibers}, they lie in the same $\Gamma$-orbit. For a general $v$, its stabilizer in $\Gamma$ is trivial, so there is a unique element $\gamma_v\in \Gamma$ with $\gamma_v v=\xi_q v$; here $\xi_q$ is a $q$-th root of unity. Since $\Gamma$ is finite, $\gamma_v=\xi_q I_V$ on a dense open subset of $V$. Hence $\gamma_v=\xi_q I_V$, and therefore $\mu_q\subset \Gamma$.

Conversely, if $\lambda I_V\in \Gamma$, then $F_i(\lambda v)=\lambda^q F_i(v)=F_i(v)$ for all $i$ and $v$. Taking $v$ and $F_i$ such that $F_i(v)\neq 0$, we get $\lambda^q=1$. It follows that $\Gamma\cap k^*I_V=\mu_q$.

We claim that the common zero of $F_0,\ldots,F_{n}$ is just the origin. Indeed, for any $v\neq 0$, we may choose a linear form $\ell\in V^*$ such that no element of the orbit $\Gamma v$ lies in $\{\ell=0\}$. Then we have $H_\ell:=\prod_{\gamma\in \Gamma}\ell\circ\gamma^{-1}\in k[V]^\Gamma=k[F_0,\dots,F_n]$ satisfies $H_\ell(v)\neq 0$. Hence $v$ is not a common zero for $\{F_i\}_{0\leq i\leq n}$. This proves the claim. Hence $f$ is a morphism of degree $q^n$.

For any $v,w\in V$ such that $f([v])=f([w])$, there exists $c\in k^*$ such that $F_i(w)=cF_i(v)$ for all $i$. Choose $\lambda\in k^*$ with $\lambda^q c=1$. Then $F_i(\lambda w)=F_i(v)$ for all $i$. By Lemma~\ref{lem:quotient-fibers}, $\lambda w=\gamma v$ for some $\gamma\in \Gamma$. Hence $[w]=[\gamma v]$ in $\mathbf{P}(V)$. Consequently, if we factor $f$ (which is $\Gamma$ or precisely $G$-equivariant) naturally as $$f = f_2 \circ f_1 \colon\  \mathbf{P}^n \xrightarrow{f_1} \mathbf{P}^n/G \xrightarrow{f_2} \mathbf{P}^n,$$ where $G = \Gamma / (\Gamma \cap k^* I_V) = \Gamma / \mu_q$, then $f_2$ is injective at closed points. Thus, $f_2$ is birational (and finite). It follows that $f_2$ is an isomorphism and $f$ is a Galois covering with Galois group $G$, which completes the proof of (4).
\end{proof}

\begin{proof}[Proof of Theorem~\ref{thm:structure-intro} and Remark \ref{rem: irred classif}]
Assume first that $f$ is Galois with Galois group $G$. We may assume that $f([v])=[F_0(v):\dots :F_n(v)]$ for each $v\in V$, where $F_i\in k[V]$ are homogeneous polynomials of degree $q$. By Lemma~\ref{lem:lifting}, there exists a finite pseudoreflection group $\Gamma\subset \mathrm{GL}(V)$ such that $k[V]^\Gamma=k[F_0,\ldots,F_{n}]$ and the Galois group of $f$ is $\Gamma/\mu_q$, where $\mu_q=\{\lambda I_V\mid \lambda^q=1\}$. We have a decomposition of the pseudoreflection representation $V$ of $\Gamma$ into irreducible summands $$V=V_1\oplus\cdots\oplus V_s, \qquad \Gamma=\Gamma_1\times\cdots\times \Gamma_s,$$ where each pseudoreflection component group $\Gamma_i\subset \mathrm{GL}(V_i)$ is irreducible. Then $k[V]^\Gamma\cong k[V_1]^{\Gamma_1}\otimes_{k}\cdots\otimes_{k}k[V_s]^{\Gamma_s},$ and the multiset of invariant degrees of $\Gamma$ is the union of the multisets of invariant degrees of the $\Gamma_i$. In particular, every invariant degree of each $\Gamma_i$ is $q$. The list of possible irreducible factors is therefore exactly the list in Theorem~\ref{thm:equal-degree}. We may choose homogeneous basic invariants $k[V_i]^{\Gamma_i}=k[\Phi_{i,1},\ldots,\Phi_{i,r_i}]$, here $r_i=\dim V_i\in \{1,2\}$ and $\deg \Phi_{i,j}=q.$ For $v=v_1+\cdots+v_s$ with $v_i\in V_i$, the map is $$f([v])=[\Phi_{1,1}(v_1):\cdots:\Phi_{1,r_1}(v_1):\cdots: \Phi_{s,1}(v_s):\cdots:\Phi_{s,r_s}(v_s)].$$ This proves the first part of the main theorem and Remark \ref{rem: irred classif}.

Conversely, if one starts with pseudoreflection groups $\Gamma_i$ and construct $f$ as in (1) and (2) of the theorem, then Lemma~\ref{lem:lifting}(4) implies that $f$ is a finite Galois cover of $\mathbf{P}^n$ with Galois group as in (3) of the theorem. This completes the proof.
\end{proof}

As a direct corollary, we have:

\begin{proof}[Proof of Corollary~\ref{cor:odd-degree-intro}]
Let $f^*\SO_{\mathbf{P}^n}(1)\cong \SO_{\mathbf{P}^n}(q)$. We apply Theorem \ref{thm:structure-intro}. If $\deg(f)$ is odd, $q$ is odd. By Theorem~\ref{thm:structure-intro}, every rank-two block has even common degree: the imprimitive block $G(q,2,2)$ requires $q$ even, and the exceptional blocks occur only for $q=12,24,60.$ Thus no rank-two block can occur. 

If $G$ is abelian, since $G(q,2,2), G_7, G_{11}, G_{19}$ are not abelian, no rank-two block can occur in Theorem~\ref{thm:structure-intro}, otherwise $G\cong \Gamma_1\times \cdots \times \Gamma_s/\Delta_{\mu_q}$ will not be abelian.

Suppose that $G$ is nilpotent and it is not a $2$-group. Since $\{G(q,2,2) \, \mid \,$q$ \text{ is not $2$-power}\}$ and $G_7, G_{11}, G_{19}$ are not nilpotent, no rank-two block can occur in Theorem~\ref{thm:structure-intro}.

In all cases, all blocks in Theorem~\ref{thm:structure-intro} are one-dimensional. It follows that $\Gamma_1\times \cdots \times \Gamma_s$ acts diagonally on $k^{n+1}$, and under suitable coordinates, the invariant ring is $$k[x_0,\ldots,x_n]^{\Gamma_1\times \cdots \times \Gamma_s} =k[x_0^q,\ldots,x_n^q].$$ Therefore, after linear changes of coordinates, $f([x_0:\cdots:x_n])=[x_0^q:\cdots:x_n^q].$ So $f$ is a $q$-power map.
\end{proof}

\begin{example}\label{ex: nilpotent galois cover}
    Consider the Galois cover $$f:\mathbf{P}^2\to \mathbf{P}^2 \qquad [x:y:z] \mapsto [x^4+y^4:x^2y^2:z^4]$$ whose Galois group $G$ is $(G(4,2,2) \times \mu_4)/\Delta_{\mu_4}$; see \cite[Page~276]{ST54} for an explicit description of the $G(4,2,2)$-action on $k^2$. Note that $|G|=16$ and $G$ is nilpotent but not abelian. In this case, $f$ is not a power map.
\end{example}

\section{Proof of {Theorem~\ref{thm:log-cy}}}\label{sect:log-cy}

In this section, we prove Theorem~\ref{thm:log-cy} as an application of Theorem~\ref{thm:structure-intro}.

\medskip

Let $V$ be one of the blocks as in Theorem~\ref{thm:structure-intro}, let $r=\dim V\in\{1,2\}$, and let $\Gamma\subset \mathrm{GL}(V)$ be the corresponding irreducible pseudoreflection group.  For a reflecting hyperplane $H\subset V$, set
\[
        e_H:=\bigl|\{\gamma\in \Gamma\mid \gamma|_H=\id_H\}\bigr| .
\]
The pointwise stabilizer of $H$ is cyclic, since it acts faithfully on the one-dimensional quotient $W/H$.

\begin{lemma}\label{lem:block-coefficients}
Assume $k[V]^\Gamma=k[\Phi_1,\ldots,\Phi_r]$ with $\deg \Phi_j=q$ for every $j$.  Then:
\begin{enumerate}
\item $e_H\le q$ for every reflecting hyperplane $H$;
\item $\displaystyle \sum_H(e_H-1)=r(q-1)$, where the sum is over all reflecting hyperplanes of $\Gamma$.
\end{enumerate}
\end{lemma}

\begin{proof}
Fix a reflecting hyperplane $H$.  Choose linear coordinates $(x,y_2,\ldots,y_r)$ on $V$ such that $H=\{x=0\}$.  A generator of the pointwise stabilizer of $H$ acts by
\[
        x\mapsto \zeta x,
        \qquad y_j\mapsto y_j,
\]
where $\zeta$ has order $e_H$.  If a monomial $x^m y_2^{a_2}\cdots y_r^{a_r}$ appears in a homogeneous invariant polynomial of degree $q$, then $e_H$ divides $m$.  If $e_H>q$, this forces $m=0$. Hence all degree-$q$ basic invariants are independent of the normal coordinate $x$.  Since these invariants generate $k[V]^\Gamma$, this would give $\trdeg k[V]^\Gamma\le r-1$, a contradiction.  Thus $e_H\le q$.

For the second assertion, consider the affine quotient map $F_V=(\Phi_1,\ldots,\Phi_r)$ from $V$ to $\mathbf{A}^r$.
Its Jacobian determinant $J_V=\det(\partial\Phi_a/\partial z_b)$ is homogeneous of degree $r(q-1)$. Since $$\mathrm{div}(J_V)=R_{F_V}=\sum_H(e_H-1)H,$$ taking degrees gives $\sum_H(e_H-1)=r(q-1)$.
\end{proof}

After the linear changes of coordinates in Theorem~\ref{thm:structure-intro}, write
\[
        V=V_1\oplus\cdots\oplus V_s,
        \qquad
        k[V_i]^{\Gamma_i}=k[\Phi_{i,1},\ldots,\Phi_{i,r_i}],
\]
where $r_i=\dim V_i\in\{1,2\}$ and $\deg\Phi_{i,j}=q$.  For a reflecting hyperplane $H\subset V_i$, set
\[
        D_{i,H}:=
        \mathbf{P}\bigl(V_1\oplus\cdots\oplus V_{i-1}\oplus H\oplus V_{i+1}\oplus\cdots\oplus V_s\bigr)
        \subset \mathbf{P}(V).
\]
Let $e_{i,H}$ be the order of the pointwise stabilizer of $H$ in $\Gamma_i$.

\begin{proposition}\label{prop:ramification-divisor}
With the notation above, the ramification divisor of $f$ is $$R_f=\sum_i\sum_{H\subset V_i}(e_{i,H}-1)D_{i,H}.$$
Consequently, if $\Delta:=\frac{1}{q-1}R_f=\sum_{i,H}a_{i,H}D_{i,H}$, then $0\le a_{i,H}\le 1$ and
\[
        \sum_{H\subset V_i}a_{i,H}=\dim V_i
\]
for every block $V_i$.
\end{proposition}

\begin{proof}
Let $F=(F_1,\ldots,F_s):V\to \oplus_i \mathbf{A}^{r_i}$, where $F_i= (\Phi_{i,1},\ldots,\Phi_{i,r_i})$ for some homogeneous polynomials of degree $q$.
The morphism $f$ is induced by $F$ via the quotient map $V\setminus\{0\} \to \mathbf{P}^n$. For $v\ne 0$, Euler's formula gives $dF_v(v)=qF(v)$.  Hence the radial tangent direction maps to the radial tangent direction, and the differential of $f$ at $[v]$ is singular exactly when $dF_v$ is singular.  Equivalently, the pullback of the projective ramification divisor to $V\setminus\{0\}$ is the affine ramification divisor of $F$ restricted to $V\setminus\{0\}$.

The Jacobian determinant of $F$ is, up to a nonzero scalar, the product of the block Jacobians.  By the proof of Lemma~\ref{lem:block-coefficients}, its divisor is $$\sum_i\sum_{H\subset V_i}(e_{i,H}-1)
        \bigl(V_1\oplus\cdots\oplus V_{i-1}\oplus H\oplus V_{i+1}\oplus\cdots\oplus V_s\bigr).$$
Projectivizing gives the formula for $R_f$.

The inequalities $a_{i,H}\le 1$ follow from Lemma~\ref{lem:block-coefficients}.  Dividing
\[
        \sum_{H\subset V_i}(e_{i,H}-1)=(\dim V_i)(q-1)
\]
by $q-1$ gives the last assertion.
\end{proof}

We now check the local singularities of the divisor in Proposition~\ref{prop:ramification-divisor}.

\begin{lemma}\label{lem:central-lines}
Let $L_1,\ldots,L_m\subset \mathbf{A}^2$ be distinct lines through the origin, and let $b_1,\ldots,b_m\in [0,1]\cap \mathbb Q$ satisfy $\sum_j b_j\le 2$. Then $(\mathbf{A}^2,\sum_{j=1}^m b_jL_j)$ is log canonical. In particular, it is log canonical if $\sum_j b_j=2$.
\end{lemma}

\begin{proof}
The pair is simple normal crossings away from the origin. Consider the blow-up $\pi:Y=\mathrm{Bl}_0\mathbf{A}^2\to \mathbf{A}^2$ and let $E$ be the exceptional divisor and $L_j'$ the strict transform of $L_j$. We have $$\pi^*(K_{\mathbf{A}^2}+\sum_j b_jL_j)=K_Y+\sum_j b_jL_j'+((\sum_j b_j)-1)E.$$
The divisor $E+\sum_j L_j'$ has simple normal crossings: the $L_j'$ are disjoint and meet $E$ transversely at distinct points.  All coefficients in the displayed boundary are at most one, because $b_j\le 1$ and $(\sum_jb_j)-1\le 1$.  Hence the original pair is log canonical.
\end{proof}

We state the following well-known lemma without a proof.

\begin{lemma}[\cite{KM98}]\label{lem:product-lc}
Let $(X_\alpha,\Delta_\alpha)$ be finitely many log-canonical pairs with each $X_\alpha$ smooth. Then
$$(\prod_\alpha X_\alpha,
        \sum_\alpha \operatorname{pr}_\alpha^*\Delta_\alpha)$$
is log canonical.
\end{lemma}

\begin{proposition}\label{prop:lc}
The pair $(\mathbf{P}(V),\Delta)$ is log canonical, where $\Delta=\sum_i\sum_{H\subset V_i}a_{i,H}D_{i,H}$ and $a_{i,H}=\frac{e_{i,H}-1}{q-1}$.
\end{proposition}

\begin{proof}
Fix a point $p=[v]\in \mathbf{P}(V)$ and write $v=v_1+\cdots+v_s$ with $v_i\in V_i$.  Choose an index $j$ with $v_j\ne 0$ and a linear form $\ell\in V_j^*$ such that $\ell(v_j)=1$.  On the affine chart $U=\{\ell=1\}\subset \mathbf{P}(V)$, a neighbourhood of $p$ is identified with
\[
        (v_j+\ker \ell)\times \prod_{i\ne j} V_i .
\]
For $i\ne j$, the divisor $D_{i,H}\cap U$ is given by the linear equation cutting out $H\subset V_i$ in the $V_i$-factor.  For $i=j$, the same statement holds after restricting the equation of $H$ to the affine space $v_j+\ker \ell$.

If $\dim V_i=1$, then the only reflecting hyperplane is $\{0\}\subset V_i$, its stabilizer has order $q$, and its coefficient is $1$.  Locally it contributes either no divisor or one smooth coordinate hyperplane with coefficient $1$.

If $\dim V_i=2$ and $v_i\ne 0$, then $v_i$ lies on at most one reflecting line, since two distinct lines in a two-dimensional vector space meet only at the origin.  Hence this block contributes either no divisor or a single smooth divisor with coefficient at most $1$.

If $\dim V_i=2$ and $v_i=0$, then all reflecting lines of the block pass through the point.  The corresponding local factor is $(V_i,\sum_{H\subset V_i}a_{i,H}H)$. By Proposition~\ref{prop:ramification-divisor}, every $a_{i,H}\le 1$ and $\sum_{H\subset V_i}a_{i,H}=2$. Thus this factor is log canonical by Lemma~\ref{lem:central-lines}.

Therefore the germ of $(\mathbf{P}(V),\Delta)$ at $p$ is a product of smooth one-dimensional coordinate-boundary germs with coefficients at most $1$, rank-two central line-arrangement germs as above, and a smooth factor carrying no boundary.  By Lemma~\ref{lem:product-lc}, this product germ is log canonical.  Since $p$ was arbitrary, $(\mathbf{P}(V),\Delta)$ is log canonical.
\end{proof}

\begin{proof}[Proof of Theorem~\ref{thm:log-cy}]
By the ramification divisor formula, we have $$R_f\sim K_{\mathbf{P}^n}-f^*K_{\mathbf{P}^n} =(n+1)(q-1)H.$$ Thus $K_{\mathbf{P}^n}+\Delta\sim_{\mathbb{Q}} 0$, where $\Delta=\frac{1}{q-1}R_f$. It remains to prove that $(\mathbf{P}^n,\Delta)$ is log canonical.

By Theorem~\ref{thm:structure-intro} and Proposition~\ref{prop:ramification-divisor}, the boundary $\Delta$ has the form studied in Proposition~\ref{prop:lc}. Hence $(\mathbf{P}^n,\Delta)$ is log canonical.
\end{proof}

\section{Essential dimensions of Galois self-covers of projective spaces}\label{sect:ess dim}
We now prove Theorem~\ref{thm:ed-intro}. The key point is to find inside the Galois group a diagonal elementary abelian subgroup of rank $n$.

\begin{lemma}\label{lem:rank-two-involutions}
Every rank-two block in Theorem~\ref{thm:structure-intro} contains the scalar element $-I$ and a non-scalar pseudoreflection of order $2$.
\end{lemma}

\begin{proof}
The scalar element $-I$ belongs to each rank-two block by Lemma~\ref{lem:lifting}(4) and the table in Theorem~\ref{thm:structure-intro} (see also Theorem~\ref{thm:equal-degree}) derived from Shephard--Todd's classification \cite{ST54}, as the common degree $q$ is even for all rank-two blocks.

For $G(q,2,2)$, the matrix
$\left(\begin{smallmatrix}0&1\\ 1&0\end{smallmatrix}\right)$
is a non-scalar pseudoreflection of order $2$. For the exceptional blocks $G_7$, $G_{11}$ and $G_{19}$, Shephard--Todd's classification lists an order-two pseudoreflection in each case. Thus every rank-two block contains the required involutive pseudoreflection.
\end{proof}

\begin{proposition}\label{prop:max-rank-subgroup}
Let $f:\mathbf{P}^n\to \mathbf{P}^n$ be a finite Galois cover with $\deg(f)>1$, and let $G$ be its Galois group. Then there is a prime $p$ and a subgroup $A\subset G$ such that $A\cong (\mu_p)^n,$ and $A$ acts faithfully and diagonally on $\mathbf{P}^n$.
\end{proposition}

\begin{proof}
Use the notation of Theorem~\ref{thm:structure-intro}. Write $k^{n+1}=V_1\oplus\cdots\oplus V_s$ and $G=(\Gamma_1\times\cdots\times \Gamma_s)/\Delta\mu_q.$ Let $r$ be the number of rank-two blocks and let $\ell$ be the number of one-dimensional blocks. Then $n+1=2r+\ell.$

First suppose that $r=0$. Choose a prime $p\mid q$. Then $\widetilde A:=(\mu_p)^{n+1}\subset (\Gamma_1\times\cdots\times \Gamma_s)=(\mu_q)^{n+1}.$ Its image in $G$ is $A=\widetilde A/\Delta\mu_p\cong (\mu_p)^n.$ This group acts diagonally and faithfully on $\mathbf{P}^n$.

Now suppose that $r>0$. Then the structure theorem shows that $q$ is even. We take $p=2$. For a one-dimensional block, let $A_i:=\mu_2\subset \Gamma_i.$ For a rank-two block $V_i$, by Lemma~\ref{lem:rank-two-involutions}, there is a non-scalar pseudoreflection $\rho_i$ of order $2$, and we set $A_i:=\langle -I_{V_i},\rho_i\rangle\subset \Gamma_i.$ Since $-I_{V_i}$ is central and $\rho_i$ is non-scalar of order $2$, $A_i\cong \mu_2\times \mu_2$ for every rank-two block.

Set $\widetilde A:=A_1\times\cdots\times A_s \subset \Gamma_1\times\cdots\times \Gamma_s.$ Then $\widetilde A\cong (\mu_2)^{\ell+2r}=(\mu_2)^{n+1}.$ Its intersection with $\Delta\mu_q$ is exactly $\Delta\mu_2$. Hence the image of $\widetilde A$ in $G$ is $A=\widetilde A/\Delta\mu_2\cong (\mu_2)^n.$ Each $A_i$ is abelian and consists of semisimple transformations, hence is diagonalizable on $V_i$. It follows that, after choosing suitable bases, $\widetilde A$ acts diagonally on $k^{n+1}$, and therefore $A$ acts diagonally on $\mathbf{P}^n$. Since $A$ is a subgroup of $G\subset \mathrm{PGL}(n+1,k)$, it acts faithfully on $\mathbf{P}^n$.
\end{proof}

\begin{proposition}\label{prop:diagonal-quotient}
Let $A\subset \mathrm{PGL}(n+1,k)$ be a faithful diagonal subgroup with $A\cong (\mu_p)^n$ for some prime $p$. Then $\ed(\mathbf{P}^n\to \mathbf{P}^n/A)=n.$ \end{proposition}

\begin{proof}
Choose homogeneous coordinates $[x_0:\cdots:x_n]$ in which $A$ acts diagonally. Let $T=\{x_0x_1\cdots x_n\neq 0\}\subset \mathbf{P}^n.$ Then $T\cong (\Gm)^n$, with coordinates $t_i=\frac{x_i}{x_0}$ for $1\leq i\leq n$. The action of $A$ on $T$ is given by characters. Since the projective action of $A$ is faithful, the characters appearing on $t_1,\ldots,t_n$ generate the full character group $A^\vee=\Hom(A,k^*).$ Let $M\cong \Z^n$ be the character lattice. The action gives a surjection $M\twoheadrightarrow A^\vee.$ Let $L$ be its kernel. Since $A^\vee\cong (\Z/p\Z)^n$ and $M$ has rank $n$, Smith normal form gives a basis $m_1,\ldots,m_n$ of $M$ such that $L=pm_1\Z\oplus\cdots\oplus pm_n\Z$. Equivalently, after replacing $t_1,\ldots,t_n$ by suitable monomial coordinates $u_1,\dots,u_n$, one has $$K(T)=k(u_1,\dots,u_n), \qquad K(T)^A=k(u_1^p,\dots,u_n^p).$$ By Lemma~\ref{lem: Kummer extension ed}, the Kummer extension $K(T)/K(T)^A=k(u_1,\dots,u_n)/k(u_1^p,\dots,u_n^p)$ has essential dimension $n$. Hence $\ed(\mathbf{P}^n\to \mathbf{P}^n/A)=n.$ \end{proof}

\begin{proof}[Proof of Theorem~\ref{thm:ed-intro}]
Let $G$ be the Galois group of $f$. Since $f$ is Galois, $f:\mathbf{P}^n\to \mathbf{P}^n$ is equal to the quotient map $\mathbf{P}^n\to \mathbf{P}^n/G.$ By Proposition~\ref{prop:max-rank-subgroup}, there is a faithful diagonal subgroup $A\subset G$ such that $A\cong (\mu_p)^n$ for some prime $p$. Therefore $f$ factors as: $\mathbf{P}^n\to \mathbf{P}^n/A\to \mathbf{P}^n/G.$ By Proposition~\ref{prop:diagonal-quotient}, $\ed(\mathbf{P}^n\to \mathbf{P}^n/A)=n.$ By \cite[Lemma~2.2(3)]{LOZ25}, $$\mathrm{ed}(f)=\mathrm{ed}(\mathbf{P}^n\to \mathbf{P}^n/G)\geq \mathrm{ed}(\mathbf{P}^n\to \mathbf{P}^n/A)=n.$$ This proves that $f$ is incompressible.
\end{proof}

The following example shows that Galois self-covers of projective spaces may not be incompressible in positive characteristic. The construction is essentially due to \cite[Paragraph after Question~19]{KZ} (see also \cite[Example~6.3]{LOZ25}).

\begin{example}\label{ex: positive char counterex}
    Let $k$ be an algebraically closed field of characteristic $p>0$, and let $n\geq 2$. We define a $(\Z/p\Z)^n$-cover $f:\mathbf{P}^n_k\longrightarrow \mathbf{P}^n_k$ by $$f([X_0:X_1:\cdots:X_n])=[X_0^p:X_1^p-X_1X_0^{p-1}:\cdots:X_n^p-X_nX_0^{p-1}].$$ When restricting to the affine chart $X_0\neq 0$, each factor is an Artin--Schreier \'etale cover with Galois group $\Z/p\Z$. Since the essential dimension of an elementary abelian $p$-group over a field of characteristic $p$ is one (see \cite[Proposition 5]{Led}), we have $\mathrm{ed}(f)\leq \mathrm{ed}((\Z/p\Z)^n)=1$ by \cite[Theorem~3.1(c)]{BR97}.
\end{example}

\appendix
\setcounter{thm}{0}

\section{A classification free proof of Corollary \ref{cor:odd-degree-intro}(2)}\label{Appendix}

In this part, we prove that an abelian Galois self-cover of a projective space of dimension $\geq 2$ are power maps, without using the classification result of Shephard and Todd \cite{ST54}.

Recall that an element in $\mathrm{PGL}(n,k)$ is a \emph{projective pseudoreflection} if it fixes a hyperplane pointwise. A subgroup of $\mathrm{PGL}(n,k)$ is a \emph{projective pseudoreflection group} if it is generated by projective pseudoreflections.

\begin{lemma}\label{lem: pseudoreflections commutes consequences}
    Let $\sigma_1,\sigma_2\in \mathrm{GL}(n+1,k)$ be two elements of finite order. Suppose that $\pi(\sigma_1)$ is a projective pseudoreflection, where $\pi: \mathrm{GL}(n+1,k) \to \mathrm{PGL}(n+1,k)$ is the natural quotient map, and $\sigma_1\sigma_2=\epsilon \sigma_2\sigma_1$ for some $\epsilon\neq 0$. Then
    \begin{itemize}
        \item[(1)] either $\epsilon=1$ and $\sigma_1$ and $\sigma_2$ are simultaneously diagonalizable, or
        \item[(2)] $\epsilon=-1$, $n=1$ and the subgroup generated by $\sigma_1$ and $\sigma_2$ is isomorphic to $\mathbb{Z}/2\mathbb{Z}\times \mathbb{Z}/2\mathbb{Z}$ as a subgroup of $\mathrm{PGL}(2,k)$.
    \end{itemize}
    Consequently, if $G\subset \mathrm{GL}(n+1)$ ($n\ge 2$) is a finite group such that $\pi(G)$ is an abelian projectively pseudoreflection group, then $G$ is abelian and all elements in $G$ are simultaneously diagonalizable.
\end{lemma}

\begin{proof}
    By definition, after choosing suitable coordinates, we may write $\sigma_1$ as $\mathrm{diag}(a,b,\cdots,b)$, where $a\neq b$ and $a^r=b^r=1$ for some $r\in \mathbb{Z}\setminus\{0\}$. We may write $\sigma_2$ as a block matrix:
    $$\sigma_2=\begin{pmatrix}
    A_{1\times 1} & B_{1\times n} \\
    C_{n\times 1} & D_{n\times n}
\end{pmatrix}.$$  Now $\sigma_1\sigma_2=\epsilon \sigma_2\sigma_1$ implies that 
\begin{equation}\label{eq: exchange of two matrix up to scaling}
    \begin{pmatrix}
    aA_{1\times 1} & aB_{1\times n} \\
    bC_{n\times 1} & bD_{n\times n}
\end{pmatrix}=\epsilon\begin{pmatrix}
    aA_{1\times 1} & bB_{1\times n} \\
    aC_{n\times 1} & bD_{n\times n}
\end{pmatrix},
\end{equation}
hence either $A\neq 0$ and $\epsilon=1$, or $A=0$, $B\neq \mathbf{0}_{1\times n}$, $C\neq \mathbf{0}_{n\times 1}$, and $\epsilon a=b$,  $\epsilon b=a$. In either case, we have $\epsilon^2=1$.

If $\epsilon=1$, then the equality \eqref{eq: exchange of two matrix up to scaling} implies that $B$ and $C$ are both zero vectors. In particular, $\sigma_1$ and $\sigma_2$ are simultaneously diagonalizable.

If $\epsilon=-1$, then the equality \eqref{eq: exchange of two matrix up to scaling} implies that $A=0$, $D=\mathbf{0}_{n\times n}$, $B\neq \mathbf{0}_{1\times n}$, $C\neq \mathbf{0}_{n\times 1}$, and $a=-b$. Since $\sigma_2$ is represented by an invertible matrix, i.e., $\rank \sigma_2 = n+1$, we have $n=1$. Moreover, $\sigma_2^2=BC\cdot \mathrm{id}$. Hence $\langle\sigma_1,\sigma_2\rangle \leq \mathrm{PGL}(2,k)$ is isomorphic to $\mathbb{Z}/2\mathbb{Z}\times \mathbb{Z}/2\mathbb{Z}$.
\end{proof}

\begin{proposition}\label{prop: abelian without classification}
    Let $f$ be an abelian Galois self-cover of $\mathbf{P}^n$ with $n\geq 2$. Then $f$ is a power map.
\end{proposition}

\begin{proof}
Write $f^*\mathcal{O}_{\mathbf P^n}(1)=\mathcal{O}_{\mathbf P^n}(q)$ for some integer $q\geq 1$. Let $G$ be the Galois group for $f$. By Lemma~\ref{lem:lifting}, the group $G$ is a natural quotient of a pseudoreflection group, hence it is a projective pseudoreflection group. Since $G$ is abelian, by Lemma~\ref{lem: pseudoreflections commutes consequences} all elements of $G$ are represented by diagonal matrices simultaneously, up to a change of homogeneous coordinates of $\mathbf{P}^n$.

Set $H_i:=\{x_i=0\}\subset \mathbf P^n$ and $B_i:=f(H_i)$ for $0\leq i\leq n$. Each $H_i$ is $G$-stable. Moreover, because $H_i$ is $G$-stable, we have set-theoretically $f^{-1}(B_i)=H_i.$ Let $r_i$ be the order of the subgroup $G_i$ of $G$ fixing $H_i$ pointwise. Then $f^*B_i=r_i H_i$. Let $d_i$ be the degree of $B_i$ for each $i$. We have $r_i=d_iq\geq q.$ 

Since $G$ is diagonal, every codimension-one fixed locus of a nontrivial element of $G$ is one of the coordinate hyperplanes $H_i$. Therefore the ramification divisor of $f$ is $R_f=\sum_{i=0}^n (r_i-1)H_i.$ By the ramification divisor formula, we have $K_{\mathbf P^n}=f^*K_{\mathbf P^n}+R_f$, it follows that $$(n+1)(q-1)\leq \deg(R_f)=\sum_{i=0}^n(r_i-1)=(n+1)(q-1)=\deg(K_{\mathbf P^n}-f^*K_{\mathbf P^n}).$$ This forces $r_0=r_1=\cdots=r_n=q.$ Consequently, $|G_i| = q$ and $d_i=1$ for every $i$.

Since $G$ is generated by projective pseudoreflections that fix some $H_i$ pointwise, we get $G=\langle G_0,\ldots,G_n\rangle$. For each $i$, the group $G_i$ is the unique order-$q$ subgroup of the one
dimensional torus consisting of diagonal transformations acting nontrivially only on the $i$-th coordinate. Hence
$$G=\{[\mathrm{diag}(a_0,\ldots,a_n)] \mid a_i^q=1\text{ for all }i\}\subset \mathrm{PGL}(n+1).$$
Since the $q$-power map is $G$-equivariant and of degree $q^n = |G|$, it equals the 
quotient map $\mathbf P^n \to \mathbf{P}^n/G$ which is $f$.
\end{proof}

\end{document}